\documentclass{amsart}
\usepackage{amssymb,amsmath,amsthm}

\theoremstyle{theorem}
\newtheorem{theorem}{Theorem}[section]
\newtheorem{proposition}[theorem]{Proposition}

\theoremstyle{definition}
\newtheorem{definition}[theorem]{Definition}

\theoremstyle{remark}

\title{Restricted Color $n$-color Compositions}

\author{Brian Hopkins}
\address{Saint Peter's University}
\email{bhopkins@saintpeters.edu}

\author{Hua Wang}
\address{Georgia Southern University}
\email{hwang@georgiasouthern.edu}

\begin{document}

\begin{abstract}
Agarwal introduced $n$-color compositions in 2000 and most subsequent research has focused on restricting which parts are allowed.  Here we focus instead on restricting allowed colors.  After three general results, giving recurrence formulas for the cases of given allowed colors, given prohibited colors, and colors satisfying modular conditions, we consider several more specific conditions, establishing direct formulas and connections to other combinatorial objects.  Proofs are combinatorial, mostly using the notion of spotted tilings introduced by the first named author in 2012.
\end{abstract}

\maketitle

\section{Introduction}
A composition of a given positive integer $n$ is an ordered sequence of positive integers with sum $n$.  For instance, there are four compositions of 3, namely $(3)$, $(2,1)$, $(1,2)$, and $(1,1,1)$.  The summands are called parts of the composition.  Compositions are sometimes referred to as ordered partitions.

Agarwal \cite{a} introduced the concept of $n$-color compositions, where a part $\kappa$ has one of $\kappa$ possible colors, denoted by a subscript $1, \ldots, \kappa$.  There are eight $n$-colored compositions of 3, namely
$$(3_1), ( 3_2), (3_3), (2_1, 1_1), (2_2, 1_1), (1_1, 2_1), (1_1, 2_2), (1_1, 1_1, 1_1).$$
There has been much research on these objects, often with restrictions on part sizes \cite{an,cdw,g10,g12,h,s}.  Related colored compositions have arisen in many places, e.g., Andrews's $k$-compositions \cite{gea} where every part but the last has $k$ possible colors, \cite{ms} where odd parts have two possible colors, and \cite{bgw18} where a part $\kappa$ has $\binom{\kappa+d-1}{d}$ potential colors. 

The combinatorial tool called spotted tilings were introduced by the first named author in \cite{h}.  The tiling for a composition of $n$ is a $1 \times n$ rectangle broken into smaller rectangles (tiles) whose lengths correspond to the part sizes.  With colors, a part $\kappa_i$ corresponds to a $1 \times \kappa$ rectangle with a spot in square $i$.  Spotted tilings for the $n$-color compositions of 3 are shown in Figure 1.  Spotted tiling were subsequently used in \cite{cdw, ggw}.  Other combinatorial interpretations of $n$-colored compositions are based on binary sequences \cite{s}, lattice paths \cite{an}, and rooted trees \cite{sa}.
\begin{figure}[h]
\setlength{\unitlength}{.4cm}
\begin{picture}(26,5)
\thicklines
\put(1,4){\line(1,0){3}}
\put(1,4){\line(0,1){1}}
\put(1,5){\line(1,0){3}}
\put(4,5){\line(0,-1){1}}
\put(1.5,4.5){\circle*{.5}}
\put(-1,4.3){$(3_1)$}

\put(1,2){\line(1,0){3}}
\put(1,2){\line(0,1){1}}
\put(1,3){\line(1,0){3}}
\put(4,3){\line(0,-1){1}}
\put(2.5,2.5){\circle*{.5}}
\put(-1,2.3){$(3_2)$}

\put(1,0){\line(1,0){3}}
\put(1,0){\line(0,1){1}}
\put(1,1){\line(1,0){3}}
\put(4,1){\line(0,-1){1}}
\put(3.5,0.5){\circle*{.5}}
\put(-1,0.3){$(3_3)$}

\put(8,4){\line(1,0){3}}
\put(8,4){\line(0,1){1}}
\put(8,5){\line(1,0){3}}
\put(10,4){\line(0,1){1}}
\put(11,5){\line(0,-1){1}}
\put(8.5,4.5){\circle*{.5}}
\put(10.5,4.5){\circle*{.5}}
\put(5,4.3){$(2_1, 1_1)$}

\put(8,2){\line(1,0){3}}
\put(8,2){\line(0,1){1}}
\put(8,3){\line(1,0){3}}
\put(10,2){\line(0,1){1}}
\put(11,3){\line(0,-1){1}}
\put(9.5,2.5){\circle*{.5}}
\put(10.5,2.5){\circle*{.5}}
\put(5,2.3){$(2_2, 1_1)$}

\put(15,4){\line(1,0){3}}
\put(15,4){\line(0,1){1}}
\put(15,5){\line(1,0){3}}
\put(16,4){\line(0,1){1}}
\put(18,5){\line(0,-1){1}}
\put(15.5,4.5){\circle*{.5}}
\put(16.5,4.5){\circle*{.5}}
\put(12,4.3){$(1_1, 2_1)$}

\put(15,2){\line(1,0){3}}
\put(15,2){\line(0,1){1}}
\put(15,3){\line(1,0){3}}
\put(16,2){\line(0,1){1}}
\put(18,3){\line(0,-1){1}}
\put(15.5,2.5){\circle*{.5}}
\put(17.5,2.5){\circle*{.5}}
\put(12,2.3){$(1_1, 2_2)$}

\put(19,4.3){$(1_1, 1_1, 1_1)$}
\put(23.25,4){\line(1,0){3}}
\put(23.25,4){\line(0,1){1}}
\put(23.25,5){\line(1,0){3}}
\put(24.25,4){\line(0,1){1}}
\put(25.25,4){\line(0,1){1}}
\put(26.25,5){\line(0,-1){1}}
\put(23.75,4.5){\circle*{.5}}
\put(24.75,4.5){\circle*{.5}}
\put(25.75,4.5){\circle*{.5}}

\end{picture}
\caption{The 8 spotted tilings for the $n$-color compositions of 3.}
\end{figure}
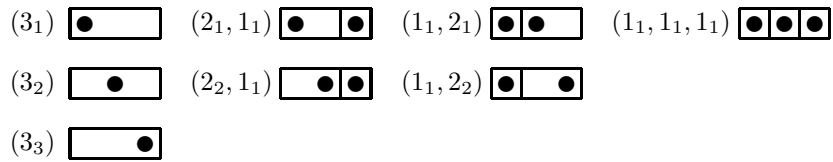

In this paper, we consider $n$-color compositions with restrictions on colors rather than part sizes.  Even and odd colors and some variations were considered in 2017 \cite{sa}, and certain of what we call modular color restrictions (see Theorem \ref{modular} below) were considered in 2019 \cite{s19}; the current work adds significantly to these explorations.  Section 2 gives general results on the enumeration of color restricted $n$-color compositions in terms of allowed or prohibited colors.  Section 3 treats many particular cases, adding combinatorial interpretations to several integer sequences and establishing various identities in terms of $n$-color compositions.  Proofs throughout are combinatorial, usually based on spotted tilings.

\section{Main Results}
We have the following very general results giving recurrence relations for the number of $n$-color compositions when a set of colors (finite or infinite) is allowed or prohibited.  There is also a result for colors satisfying certain modular conditions.
  
We follow the convention that there is an empty composition with sum zero, so that $a(0) = 1$ in most of our recursively defined sequences.  Also, in this section only, we use superscripts to index a list of numbers in order to avoid confusion with the subscripts denoting colors.  The sequence $c^1, c^2, \ldots$ is abbreviated $\{c^i\}$.  Let $c^- = \min(\{c^i\})$ and $c^+ = \max(\{c^i\})$.  We also use $A \setminus B$ for the complement of $B$ inside $A$. 

\subsection{Allowed Colors}
We begin with the situation where only specified colors are allowed.  

\begin{theorem}\label{allowed}
For allowed colors $\{c^i\}$, the number of color restricted $n$-color compositions of $n$ is given by
\[ a(n) = a(n-1) + \sum_i a(n-c^i) \]
with $a(n) = 0$ for $n \leq -1$, $a(0) = 1$, $a(1) = \cdots = a(c^- - 1) = 0$ (if $c^- \ne 1$), and $a(c^-) = 1$.
\end{theorem}

\begin{proof}
We build all allowed compositions of $n$ from the compositions counted by the summands in the right-hand side of the equation.  
\begin{enumerate}
\item For a composition of $n-1$, increase its last part, say $\kappa_c$, to $(\kappa+1)_c$ to obtain an allowed composition of $n$.
\item For a composition of $n-c$ for some $c \in \{c^i\}$, add a part $c_c$ to make an allowed composition of $n$.
\end{enumerate}
These compositions are distinct since those from (1) have an empty last square while those from (2) have a spot in the last square, and the compositions from (2) are distinct since the added last part is different for each $c^i$.  

The reverse map is clear: If the last tile has a spot in the last square, then remove that tile.  Otherwise, decrease the last tile by one square.

The initial values through $a(c^-) = 1$ are clear.  For other $n=c \in \{c^i\}$, the sequence term $a(n)$ calls step (2) of the bijection, generating $(n_n)$ from the empty composition in addition to any compositions consisting of parts with smaller colors.
\end{proof}

See section 3.1 for several examples.

Note that when $c^-=1$, there will be two $a(n-1)$ terms on the right-hand side: The bijection step (1) produces $a(n-1)$ compositions of $n$ with empty last squares while (2) applied to $1 \in \{c^i\}$ produces the $a(n-1)$ compositions of $n$ ending in $1_1$.

\subsection{Prohibited Colors}
For situations where certain colors are prohibited, Theorem \ref{allowed} would suffice using the complementary set of allowed colors.  However, for a finite list of prohibited colors, that would give an infinite recurrence.  Theorem \ref{prohibited} establishes a finite recurrence given a finite list of prohibited colors.

\begin{theorem}\label{prohibited}
For prohibited colors $\{d^i\}$, the number of color restricted $n$-color compositions of $n$ is given by
\[ a(n) = 3a(n-1) - a(n-2) + \sum_i \left( -a(n-d^i) + a(n-d^i-1) \right) \]
with initial values $a(0), \ldots, a(d^+)$ provided by Theorem \ref{allowed} applied to the colors $\{1, \ldots, d^+\} \setminus \{d^i\}$.
\end{theorem}

\begin{proof}
For the bijection, we rewrite the recurrence as
\[ a(n) +a(n-2) + \sum_i a(n-d^i) = 3a(n-1) + \sum_i a(n-d^i-1). \]
We proceed from the compositions counted by terms in the left-hand side.
\begin{itemize}
\item[1a.] For a composition of $n$ whose last part is $\kappa_j$ for some $j < \kappa$, decrease the last part to $(\kappa-1)_j$ to obtain an allowed composition of $n-1$.
\item[1b.] For a composition of $n$ with last part $1_1$, remove $1_1$ to obtain an allowed compositions of $n-1$.
\item[1c.] For a composition of $n$ whose last part is $\kappa_\kappa$ for some $\kappa \ge 2$ with $\kappa \neq d^i+1$ for any $i$, replace $\kappa_\kappa$ with $(\kappa-1)_{\kappa-1}$ to obtain an allowed composition of $n-1$.
\item[1d.] For compositions of $n$ whose last part is $\kappa_\kappa$ for some $\kappa \ge 2$ with $\kappa= d^i+1$ for some $i$, remove $\kappa_\kappa$ to obtain an allowed composition of $n-d^i-1$.
\item[2.] For a composition of $n-2$, increase its last part $\kappa_j$ to $(\kappa+1)_j$ to obtain an allowed composition of $n-1$.
\item[3.] For any $i$ and a composition of $n-d^i$, add a part $(d^i-1)_{d^i-1}$ to obtain an allowed compositions of $n-1$.
\end{itemize}

Figure \ref{fig:bij1} illustrates these aspects of the bijection.

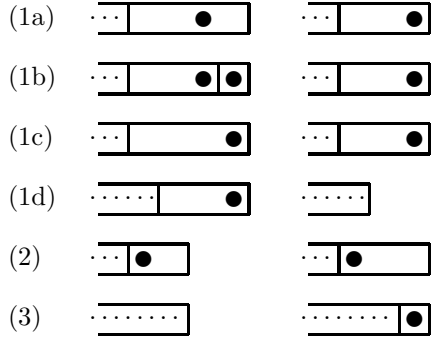
\begin{figure}[h]
\setlength{\unitlength}{.4cm}
\begin{picture}(12,12)
\thicklines

\put(-1,11){\line(1,0){5}}
\put(-1,12){\line(1,0){5}}
\put(0,11){\line(0,1){1}}
\put(4,11){\line(0,1){1}}
\put(2.5,11.5){\circle*{.5}}
\put(-1.3,11.5){$\ldots$}
\put(-4,11.3){(1a)}

\put(6,11){\line(1,0){4}}
\put(6,12){\line(1,0){4}}
\put(7,11){\line(0,1){1}}
\put(10,11){\line(0,1){1}}
\put(9.5,11.5){\circle*{.5}}
\put(5.7,11.5){$\ldots$}

\put(-1,9){\line(1,0){5}}
\put(-1,10){\line(1,0){5}}
\put(0,9){\line(0,1){1}}
\put(3,9){\line(0,1){1}}
\put(4,9){\line(0,1){1}}
\put(2.5,9.5){\circle*{.5}}
\put(3.5,9.5){\circle*{.5}}
\put(-1.3,9.5){$\ldots$}
\put(-4,9.3){(1b)}

\put(6,9){\line(1,0){4}}
\put(6,10){\line(1,0){4}}
\put(7,9){\line(0,1){1}}
\put(10,9){\line(0,1){1}}
\put(9.5,9.5){\circle*{.5}}
\put(5.7,9.5){$\ldots$}

\put(-1,7){\line(1,0){5}}
\put(-1,8){\line(1,0){5}}
\put(0,7){\line(0,1){1}}
\put(4,7){\line(0,1){1}}
\put(3.5,7.5){\circle*{.5}}
\put(-4,7.3){(1c)}
\put(-1.3,7.5){$\ldots$}

\put(6,7){\line(1,0){4}}
\put(7,7){\line(0,1){1}}
\put(6,8){\line(1,0){4}}
\put(10,8){\line(0,-1){1}}
\put(9.5,7.5){\circle*{.5}}
\put(5.7,7.5){$\ldots$}

\put(-1,5){\line(1,0){5}}
\put(-1,6){\line(1,0){5}}
\put(1,5){\line(0,1){1}}
\put(4,5){\line(0,1){1}}
\put(3.5,5.5){\circle*{.5}}
\put(-4,5.3){(1d)}
\put(-1.3,5.5){$\ldots \ldots$}

\put(6,5){\line(1,0){2}}
\put(6,6){\line(1,0){2}}
\put(8,5){\line(0,1){1}}
\put(5.7,5.5){$\ldots \ldots$}

\put(6,1){\line(1,0){4}}
\put(6,2){\line(1,0){4}}
\put(9,1){\line(0,1){1}}
\put(10,1){\line(0,1){1}}
\put(9.5,1.5){\circle*{.5}}
\put(-4,1.3){(3)}
\put(5.7,1.5){$\ldots . \ldots .$}

\put(-1,1){\line(1,0){3}}
\put(-1,2){\line(1,0){3}}
\put(2,1){\line(0,1){1}}
\put(-1.3,1.5){$\ldots . \ldots . $}

\put(-1,3){\line(1,0){3}}
\put(-1,4){\line(1,0){3}}
\put(0,3){\line(0,1){1}}
\put(2,3){\line(0,1){1}}
\put(.5,3.5){\circle*{.5}}
\put(-4,3.3){(2)}
\put(-1.3,3.5){$\ldots$}

\put(6,3){\line(1,0){4}}
\put(6,4){\line(1,0){4}}
\put(7,3){\line(0,1){1}}
\put(10,3){\line(0,1){1}}
\put(7.5,3.5){\circle*{.5}}
\put(5.7,3.5){$\ldots$}

\end{picture}

\caption{The $n$-color compositions with the spotted tilings on the left-hand side are mapped to the corresponding compositions on their right by the bijection in the proof of Theorem \ref{prohibited}, demonstrating each case when $d^1=2$.}\label{fig:bij1}
\end{figure}

As suggested by the labels, maps (1a) through (1d) treat all allowed compositions of $n$, map (2) treats all allowed compositions of $n-2$, and map (3) handles all allowed compositions of $n-d^i$ for the prohibited colors $\{d^i\}$.  Note that (1b) handles the case $d^- \ne 1$; see the note following the proof for the simplified bijection when $d^- = 1$.

Over all allowed compositions of $n$, (1a) produces a complete set of allowed compositions of $n-1$, as does (1b).  The third set of allowed compositions of $n-1$ comes from (1c), (2), and (3): All compositions with an empty final square are produced by (2) applied to all allowed compositions of $n-2$.  For compositions with a spot in the final square, for all $i$, (3) applied to all allowed compositions of $n-d^i$ produces the compositions with final part $(d^i-1)_{d^i-1}$ while (1c) produces all the others.  (Note that the $\kappa \neq d^i+1$ restriction in (1c) guarantees that the replacement part $(\kappa-1)_{\kappa-1}$ uses an allowed color and, since $\kappa \notin \{d^i\}$, it does not duplicate the output of (3).)  Finally, (1d) produces all allowed compositions of $n-d^i-1$ over all $i$.

For the initial values, for $1 \le n \le d^+$, prohibiting the colors $\{d^i\}$ is equivalent to allowing the complementary colors $\{1, \ldots, d^+\} \setminus \{d^i\}$, so that the initial sequence generated by Theorem \ref{allowed} is valid.  In the case that $\{d^i\} = \{1, \ldots, d^+\}$, this leads to $a(0) = 1$, $a(1) = \cdots = a(d^+) = 0$.  In any case, the sequence term for $n=d^++1$ involves step (1d) of the bijection, connecting $(n_n)$ and the empty composition.
\end{proof}

See section 3.2 for further examples.

If $d^-=1$, then the recurrence becomes
\[a(n) + a(n-2) + \sum_{i\ne1} a(n-d^i) = 2a(n-1) + \sum_i a(n-d^i-1)\]
and the bijection simplifies: Step (1b) never occurs, leaving two sets of allowed compositions counted by $n-1$.  Also note that the $d^-=1$ case is excluded from (3) to avoid producing compositions of $n-1$ with last part ``$0_0$.''

Also, whenever there are sequential forbidden colors, terms in the recurrence relation of Theorem \ref{prohibited} cancel out.  For instance, for forbidden colors $k, k+1, \dots, k+d$, the recurrence relation reduces to 
\[a(n) = 3a(n-1) - a(n-2) -a(n-k) + a(n-k-d-1).\]

\subsection{Modular Colors}
When there are infinitely many colors both allowed and prohibited, there can still be a finite recurrence relation for the number of resulting compositions.  Our final main theorem addresses such a situation, where all colors having the same modulus are either all allowed or all prohibited.

\begin{theorem}\label{modular}
For allowed colors congruent to $\{m^i\}$ modulo $m$, the number of color restricted $n$-color compositions of $n$ is given by
\[ a(n) = a(n-1) + a(n-m) - a(n-m-1) + \sum_i a(n-m^i)\]
where $1 \leq m^i \leq m$ for each $i$.
The initial values $a(0), \ldots, a(m+1)$ are provided by Theorem \ref{allowed} applied to the colors $\{m^i\}$ if $m^- \ne 1$, to colors $\{m^i\} \cup \{m+1\}$ if $m^- = 1$.
\end{theorem}

\begin{proof}
For the bijection, we rewrite the recurrence as
\[ a(n) +a(n-m-1) = a(n-1) + a(n-m) + \sum_i a(n-m^i). \]

We proceed from the compositions counted by terms in the left-hand side.
\begin{itemize}
\item[1a.] For a composition of $n$ whose last part is $\kappa_j$ for some $j < \kappa$, decrease the last part to $(\kappa-1)_j$ to obtain an allowed composition of $n-1$.
\item[1b.] For a composition of $n$ whose last part is $\kappa_\kappa$ with $\kappa \le m$, then $\kappa=m^i$ for some $i$; remove $\kappa_\kappa$ to obtain an allowed composition of $n-m^i$.
\item[1c.] For a composition of $n$ whose last part is $\kappa_\kappa$ with $\kappa \geq m+1$, replace $\kappa_\kappa$ with $(\kappa-m)_{(\kappa-m)}$ to obtain an allowed composition of $n-m$.
\item[2.] For a composition of $n-m-1$, increase its last part, say $\kappa_j$, to $(\kappa+1)_j$ to obtain an allowed composition of $n-m$.
\end{itemize}

Figure \ref{fig:bij2} illustrates these aspects of the bijection.

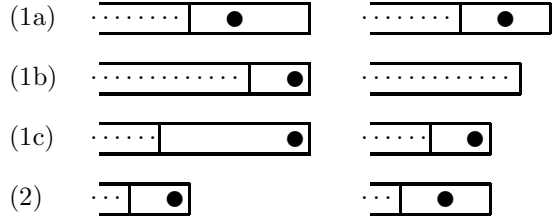
\begin{figure}[t]
\setlength{\unitlength}{.4cm}
\begin{picture}(13,8)
\thicklines

\put(-1,7){\line(1,0){7}}
\put(-1,8){\line(1,0){7}}
\put(2,7){\line(0,1){1}}
\put(6,7){\line(0,1){1}}
\put(3.5,7.5){\circle*{.5}}
\put(-4,7.3){(1a)}
\put(-1.3,7.5){$\ldots . \ldots .$}

\put(8,7){\line(1,0){6}}
\put(8,8){\line(1,0){6}}
\put(11,7){\line(0,1){1}}
\put(14,7){\line(0,1){1}}
\put(12.5,7.5){\circle*{.5}}
\put(7.7,7.5){$\ldots . \ldots .$}

\put(-1,5){\line(1,0){7}}
\put(-1,6){\line(1,0){7}}
\put(4,5){\line(0,1){1}}
\put(6,5){\line(0,1){1}}
\put(5.5,5.5){\circle*{.5}}
\put(-4,5.3){(1b)}
\put(-1.3,5.5){$\ldots\ldots\ldots\ldots .$}

\put(8,5){\line(1,0){5}}
\put(8,6){\line(1,0){5}}
\put(13,5){\line(0,1){1}}
\put(7.7,5.5){$\ldots\ldots\ldots\ldots .$}

\put(-1,3){\line(1,0){7}}
\put(-1,4){\line(1,0){7}}
\put(1,3){\line(0,1){1}}
\put(6,3){\line(0,1){1}}
\put(5.5,3.5){\circle*{.5}}
\put(-4,3.3){(1c)}
\put(-1.3,3.5){$\ldots\dots$}

\put(8,3){\line(1,0){4}}
\put(10,3){\line(0,1){1}}
\put(8,4){\line(1,0){4}}
\put(12,4){\line(0,-1){1}}
\put(11.5,3.5){\circle*{.5}}
\put(7.7,3.5){$\ldots\ldots$}

\put(-1,1){\line(1,0){3}}
\put(-1,2){\line(1,0){3}}
\put(0,1){\line(0,1){1}}
\put(2,1){\line(0,1){1}}
\put(1.5,1.5){\circle*{.5}}
\put(-4,1.3){(2)}
\put(-1.3,1.5){$\ldots$}

\put(8,1){\line(1,0){4}}
\put(8,2){\line(1,0){4}}
\put(9,1){\line(0,1){1}}
\put(12,1){\line(0,1){1}}
\put(10.5,1.5){\circle*{.5}}
\put(7.7,1.5){$\ldots$}

\end{picture}
\caption{Demonstration of each case of the bijection in the proof of Theorem \ref{modular}, with $m=3$ and $m^1=2$.}\label{fig:bij2}
\end{figure}

Similar to before, maps (1a), (1b), and (1c) treat all allowed compositions of $n$ and map (2) treats all allowed compositions of $n-m-1$.

Over all allowed compositions of $n$, (1a) produces all allowed compositions of $n-1$, (1b) produces all allowed compositions of $n-m^i$ for all $\{m_i\}$ with $1 \leq m^i \leq m$, and (1c) produces all allowed compositions of $n-m$ with a spot in the last square (note that if $\kappa$ is an allowed colors, then so is $\kappa-m$).  Finally, over all allowed compositions of $n-m-1$, (2) produces all allowed compositions of $n-m$ with empty last squares.

For the initial values, for $1 \le n \le m+1$, allowing colors congruent to $\{m^i\}$ modulo $m$ is just allowing the colors $\{m^i\}$ with $1 \leq m^i \leq m$ for each $i$ if $m^- \ne 1$, or the colors $\{m^i\} \cup \{m+1\}$ if $m^- = 1$.  Larger colors are handled by the bijection.
\end{proof}

The special treatment of color $m+1$ in the proof, should it be included, comes from step (2) of the bijection which at $n=m+1$ would call for sending ``$0_0$'' to ``$1_0$.''  Addressing $n=m+1$ in the setting of Theorem \ref{allowed} means that the bijection here will only be applied to colors $n=m+2$ or greater where compositions of $n-m-1$ have nonzero parts with nonzero colors.

See section 3.3 for additional examples.

\section{Examples and Additional Results}
The majority of this article is dedicated to establishing further results for particular classes of allowed or prohibited colors, such as direct formulas and bijections to other classes of compositions or various combinatorial objects.  The three subsections mirror the three theorems above.  Most results are general cases. There are also a handful of more specific results, e.g., prohibiting the color 2; see \cite{hw} for additional specific results.  We will use the following definitions in describing several of the bijections.

\begin{definition}
Given a tile $\kappa_c$ of an $n$-color composition, the $c$-block consists of the first $c$ squares: starting from the left, $c-1$ empty squares and the spotted square indicating the tile's color.  The $c$-block is followed by a tail of $\kappa-c$ empty squares; if $\kappa = c$ then we say the tail is empty.
\end{definition}

\subsection{Allowed colors}
We begin with the examples having a finite number of allowed colors.  The recurrences follow from Theorem \ref{allowed}.  In each of the following general cases, we derive a closed formula and, in most, present bijections to other combinatorial objects.  See \cite{hw} for further particular examples.

\subsubsection{Single colored compositions}
First we consider allowing parts of only one color.  The direct formula in the next result is a generalization of the following classical result: The number of regular compositions of $n$ with $m$ parts is the binomial coefficient $\binom{n-1}{m-1}$.  The binomial theorem then gives one way (of many) to conclude that there are $2^{n-1}$ regular compositions of $n$.  Note that only allowing the color $c=1$ is equivalent to regular compositions; simply ignore the color/spot.

\begin{proposition} \label{1colorform}
The number of $n$-color compositions of $n$ with only the one color $c > 1$ allowed is given by the recurrence relation
\[a(n) = a(n-1) + a(n-c)\]
with $a(n) = 0$ for $n \leq -1$, $a(0) = 1$, and $a(1) = \cdots = a(c-1) = 0$.  Also,
\[a(n) = \sum_{m=1}^n {n-(c-1)m - 1 \choose m-1}.\]
\end{proposition}
\begin{proof}
The recurrence comes from Theorem \ref{allowed}.

For the direct formula, consider the compositions (and hence the corresponding spotted tiling representations) of $n$, allowing color $c$, with $m$ parts. The entire tiling is an arrangement of $m$ $c$-blocks and $n-mc$ empty squares that constitute the tails. The number of ways to assign $n-mc$ empty squares to the tails of $m$ tiles is
$$ { n - mc + m -1 \choose m-1 }$$
equivalent to the binomial coefficient above.  Summing over possible values of $m$ gives the formula.
\end{proof}

The counting sequences for small cases include the Fibonacci numbers for $c=2$, the ``tribonacci'' numbers for $c= 3$, \cite[A003269]{oeis} for $c=4$, and \cite[A003520]{oeis} for $c=5$. Such sequences also enumerate compositions with constraints on part sizes as detailed in the next result.

\begin{proposition} \label{1colorcomb}
Given positive integers $n$ and $c$, the following three combinatorial objects are equinumerous.
\begin{enumerate}
\item[(i.)] $n$-color compositions of $n$ with only color $c$ allowed,
\item[(ii.)] regular compositions of $n$ with parts $c$ or greater,
\item[(iii.)] regular compositions of $n$ with parts 1 or $c$ and first part $c$.
\end{enumerate}
\end{proposition}

\begin{proof}
We relate each type of restricted uncolored compositions to (i).  The connection with (ii) is clear:
A part with color $c$ has size at least $c$, so removing the color from the parts of the $n$-color composition (equivalently, removing the spots from the corresponding spotted tiling) leaves a regular composition of the desired type.  For the reverse direction, simply add the color $c$ to each part/add a spot in square $c$ of each tile.

To connect (i) and (iii), consider the spotted tiling of a composition with color $c$ and the following bijection.  
\begin{enumerate}
\item In each tile, map the $c$-block to a part $c$.
\item Map each square in the tail to a part 1.
\end{enumerate}
For the reverse map, combine each $c$ with any subsequent parts 1 to make a tile $\kappa_c$. Figure 4
shows an example.
\end{proof}
\begin{figure}[t]
\setlength{\unitlength}{.4cm}
\begin{picture}(12,4)
\thicklines

\put(0,3){\line(1,0){12}}
\put(0,4){\line(1,0){12}}
\put(0,3){\line(0,1){1}}
\put(5,3){\line(0,1){1}}
\put(2.5,3.5){\circle*{.5}}
\put(8,4){\line(0,-1){1}}
\put(12,4){\line(0,-1){1}}
\put(7.5,3.5){\circle*{.5}}
\put(10.5,3.5){\circle*{.5}}

\put(0,1){\line(1,0){12}}
\put(0,2){\line(1,0){12}}
\put(0,1){\line(0,1){1}}
\put(3,1){\line(0,1){1}}
\put(4,1){\line(0,1){1}}
\put(5,1){\line(0,1){1}}
\put(8,1){\line(0,1){1}}
\put(11,1){\line(0,1){1}}
\put(12,1){\line(0,1){1}}

\end{picture}
\caption{An example of the bijection in Proposition \ref{1colorcomb} for $n=12$ and $c=3$.  The spotted tiling of the $n$-colored composition $(5_3, 3_3, 4_3)$ is mapped to the regular composition tiling for $(3,1,1,3,3,1)$ below it.}\label{fig:bij1c}
\end{figure}
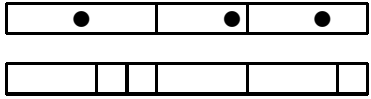

It is interesting to consider Proposition \ref{1colorcomb} when $c=1$.  This allows one possibility for every part size which is equivalent to regular compositions (simply erase the spots), so the connection between (i) and (ii) is clear.
What is described by (iii) in this case?  There is only one composition of $n-1$ where every part is a 1, namely $(1, \ldots, 1)$ with $n-1$ parts.  Here we need to understand ``1 or $c$'' with $c=1$ as indicating two types of 1s, say 1 and $\overline{1}$.  Indeed, there are $2^{n-1}$ compositions of $n-1$ with each part 1 or $\overline{1}$, as there are two choices for each of the $n-1$ parts.

\subsubsection{Compositions with two colors}
Next, we consider cases where two arbitrary colors are allowed.  We give two results analogous to, but a little more complicated than, Propositions \ref{1colorform} and \ref{1colorcomb} above.  In the following proposition, we follow the convention that $\binom{m}{x} = 0$ for noninteger $x$.

\begin{proposition} \label{2colorform}
The number of $n$-color compositions of $n$ with colors $b$ and $c$ (where $b<c$) is given by
\[a(n) = a(n-1) + a(n-b) + a(n-c)\]
with $a(n) = 0$ for $n \leq -1$, $a(0) = 1$, and $a(1) = \cdots = a(b-1) = 0$ if $b>1$.  Also,
\[a(n)=\sum_{m=1}^n \sum_{i=0}^{n-bm} {i+m-1 \choose m-1}  {m \choose (n-bm-i)/(c-b)}.  \]
\end{proposition}
\begin{proof}
The recurrence follows directly from Theorem \ref{allowed}.

For the direct formula, let $m$ be the number of parts.  Since only colors $b$ and $c$ are allowed, each spot is preceded by at least $b-1$ empty squares; these and the spots account for $bm$ squares.  These are either $b$-blocks or incomplete $c$-blocks in need of additional empty squares.  The remaining $n-bm$ squares either contribute to tails (empty squares after spots to make greater parts) or complete $c$-blocks (empty squares before the spot).  

More precisely, let $i$ be the total number of squares in the tails of parts.  There are ${i+m-1 \choose m-1}$ ways of partitioning these $i$ squares into the tails of $m$ tiles.  The remaining $n-b m - i$ squares need to be distributed before spots in some tiles.  For each tile, we have already accounted for $b-1$ empty squares before the spot; if additional empty squares are added, then there must be exactly $c-b$ of them (to make the part have color $c$ rather than $b$).  This is only possible when $n-bm - i$ is divisible by $c-b$, in which case there are ${m \choose (n-bm-i)/(c-b)}$ ways to select the tiles to extend to $c$-blocks.  (Note that, by the binomial coefficient convention, nonzero contributions to the sum arise only arise when $n-bm-i$ is a multiple of $c-b$.)
\end{proof}

\begin{proposition} \label{2colorcomb}
Given positive integers $n$ and $b, c$ with $b<c$, the following three combinatorial objects are equinumerous.
\begin{enumerate}
\item[(i.)] $n$-color compositions of $n$ with only colors $b$ and $c$ allowed,
\item[(ii.)] regular compositions of $n$ with parts $b$ or greater where there are two types of parts $c$ and greater,
\item[(iii.)] regular compositions of $n$ with parts 1, $b$, or $c$ with first part $b$ or $c$.
\end{enumerate}
\end{proposition}

\begin{proof}
The proof is fairly similar to that of Proposition \ref{1colorcomb}.

To connect the $n$-color compositions of (i) to the regular compositions of (ii), it suffices to note that every part must be at least $b$ (since the smallest allowed color is $b$), and parts $c$ or greater each have two possible types corresponding to the allowed colors $b$ and $c$.

With the following bijection connects the $n$-color compositions of (i) to the union of regular compositions given in (iii).  
\begin{enumerate}
\item In each tile, map the $b$- or $c$-block to a part $b$ or $c$, respectively.
\item Map each square in the tail to a part 1.
\end{enumerate}

For the reverse map, combine each $b$ or $c$ with any subsequent parts 1 to make a tile $\kappa_b$ or $\kappa_c$, respectively.
\end{proof}

There are further results for the case of consecutive colors ($c = b+1$) and several connections to well known integer sequences for small values of $b$; see \cite{hw} for details.  The interested reader could formulate and prove analogous statements to Propositions \ref{2colorform} and \ref{2colorcomb} for three or more arbitrary colors.

\subsubsection{Compositions with colors $1,2,\ldots,c$}
We conclude this subsection by considering consecutive allowed colors starting with 1 and continuing to an arbitrary $c \ge 1$, a case considered in the first paper on $n$-color compositions \cite[\S 3]{a}.  The formula in the following enumeration result combines to $c$ nested summations; after the proof we expand the formulas for small $c$.

\begin{proposition}\label{prop:allow1d} 
The number of $n$-color compositions of $n$ allowing colors $1,2,\ldots,c$ is given by 
\[a(n) = 2a(n-1) + a(n-2) + \cdots + a(n-c)\]
with $a(0) = a(1) = 1$.  Also, 
$$ a(n) = \sum_{i_c=1}^n \sum_{\ell=0}^{n-i_c}  {\ell+i_c-1 \choose i_c-1} G_c(n_c, i_c)  $$
where 
$$ G_j(n_j, i_j) = \sum_{i_{j-1}=0}^{i_j}  {i_j \choose i_j-i_{j-1}}  G_{j-1}(n_{j-1}, i_{j-1}) $$
for $3 \le j \le c$ and $G_2(n_2, i_2) = {i_2 \choose n_2 }$ where $n_c = n-i_c-\ell$ and, for $ 2\le j \le c-1$, $n_j = n_{j+1} - j(i_{j+1}-i_j)$.
\end{proposition}

\begin{proof}  The recurrence follows directly from Theorem \ref{allowed}.

For the direct formula, let $i_c$ be the number of parts, equivalently, the number of spotted squares in the tiling.  Also, let $\ell$ be the total number of squares in the tails of the tiles.   There are ${\ell+i_c-1 \choose i_c-1}$ ways of partitioning these $\ell$ squares into the tails of $i_c$ tiles.  It remains to distribute the other $n_c = n-i_c-\ell$ squares as the empty squares preceding spotted squares in the $k$-blocks for tiles with color $k$ (with $k \le c$); suppose there are $G_c(n_c, i_c)$ ways to do this.  Then the number of $n$-color compositions of $n$ allowing colors $1,2,\ldots,c$ is
$$ a(n) = \sum_{i_c=1}^n  \sum_{\ell=0}^{n-i_c} {\ell+i_c-1 \choose i_c-1} G_c(n_c, i_c).$$

Determining $G_c(n_c, i_c)$ essentially involves an inductive process.  We detail the next step, a general step, and the last step (which can be considered the base case).

Let $i_j$ be the number of parts with color at most $j$ (so that, in our setting, $i_c$ is the total number of parts).  Now $i_c - i_{c-1}$ is the number of parts with color $c$ and there are $\binom{i_c}{i_c-i_{c-1}}$ ways to select which parts have color $c$.  Since each color $c$ part requires $c-1$ empty squares in its $c$-block, we have $n_{c-1}=n_c - (c-1)(i_c - i_{c-1})$ other squares left to assign to the $i_{c-1}$ parts with colors at most  $c-1$.  Suppose there are $G_{c-1}(n_{c-1},i_{c-1})$ ways to do this.  Then
\[ G_c(n_c, i_c) = \sum_{i_{c-1}=0}^{i_c} \binom{i_c}{i_c - i_{c-1}} G_{c-1}(n_{c-1},i_{c-1}).\]

In general, suppose $G_j(n_j,i_j)$ is the number of ways to distribute $n_j$ empty squares preceding spots to form $k$-blocks for $i_j$ parts with colors at most $j$.  There are $i_j - i_{j-1}$ parts of color $j$ chosen among all the $i_j$ parts in  $\binom{i_j}{i_j-i_{j-1}}$ ways leaving $n_{j-1}=n_j - (j-1)(i_j - i_{j-1})$ other squares left to assign to the $i_{j-1}$ parts with color at most $j-1$; suppose there are $G_{j-1}(n_{j-1},i_{j-1})$ ways to do this.  Then
\[ G_j(n_j, i_j) = \sum_{i_{j-1}=0}^{i_j} \binom{i_j}{i_j - i_{j-1}} G_{j-1}(n_{j-1},i_{j-1}).\]

For $j=2$, the formula simplifies considerably.  The $i_2$ parts have color 1 or 2.  Each of the $n_2 = n_3 - 2(i_3 - i_2)$ empty squares precedes a spot to make the 2-block of a color 2 part.  Therefore $G_2(n_2, i_2) = {i_2 \choose n_2 }$.
\end{proof}

As illustrations of the proposition, we detail the formulas for small values of $c$.  As mentioned above, the $c=1$ case corresponds to regular compositions.  In terms of the proposition formula, it is easy to see that $\ell =n -i_1$ where $i_1$ is the number of parts. Then there are ${\ell + i_1 -1 \choose i_1-1}$ ways to distribute the empty squares among the $i_1$ tails. Summing over $i_1$ gives
$$ \sum_{i_1=1}^n {\ell + i_1 -1 \choose i_1-1} = \sum_{i_1=1}^n {n-1 \choose i_1-1} = 2^{n-1}. $$

For $c=2$, we have $G_2(n_2,i_2) = {i_2 \choose n_2}$ with $n_2 =  n-i_2-\ell$, therefore
$$ \sum_{i_2=1}^n \sum_{\ell =0}^{n-i_2} {\ell + i_2 -1 \choose i_2-1}  { i_2 \choose n-i_2-\ell}  $$
gives the number of $n$-color compositions of $n$ allowing colors 1 and 2.

When $c=3$, we have $G_3(n_3, i_3) = \sum_{i_2 = 0}^{i_3} {i_3 \choose i_3-i_2}G_2(n_2, i_2)$ with $n_3 = n - i_3 - \ell$ and $n_2 = n_3 - 2 (i_3 - i_2)$. Thus the number of $n$-color compositions of $n$ allowing colors 1, 2, and 3 is
$$ \sum_{i_3=1}^n \sum_{\ell =0}^{n-i_3} \sum_{i_2 = 0}^{i_3}  {\ell + i_3 -1 \choose i_3-1}   {i_3 \choose i_3-i_2} { i_2 \choose n-3i_3- 2 i_2 - \ell} . $$

\subsection{Prohibited colors}
Our primary applications of Theorem \ref{prohibited} concern the case where colors $1, \ldots, d$ are prohibited for an arbitrary $d \ge 1$.  The section also includes results for the specific case where the color 2 is prohibited.  

\subsubsection{Compositions prohibiting colors $1,2,\ldots,d$}
Analogous to Proposition \ref{prop:allow1d}, the next proposition gives a (simpler) counting formula.  We also provide a connection to certain regular compositions with some notes on the case $d=1$, as usual.  First, we recall a result on the enumeration of $n$-color compositions by number of parts.

\begin{theorem} \label{ccompparts}
The number of $n$-color compositions of $n$ with $m$ parts is $\binom{n+m-1}{2m-1}$.
\end{theorem}
See \cite[Thm. 1]{a} for the original proof and \cite[Thm. 4]{h} for a combinatorial proof.

\begin{proposition} 
Given $d \ge 1$, the number of $n$-color compositions of $n$ prohibiting colors $1,2,\ldots,d$ is given by
\[a(n) = 2a(n-1) - a(n-2) + a(n-d-1)\] 
with $a(0) = 1$ and $a(1) = \cdots = a(d) = 0$.  Also,
\[a(n) = \sum_{m=1}^{\lfloor n/(d+1) \rfloor } {n-(d-1)m-1 \choose 2m-1} .\]
\end{proposition}

\begin{proof}
The recurrence follows directly from Theorem \ref{prohibited}.

For the direct formula, we determine the number of desired length $n$ compositions having $m$ parts.  Note that $m$ can range from 1 to $\lfloor n/(d+1) \rfloor$ since each part must be at least $d+1$.

With the color prohibition, each part $\kappa_c$ begins with $d$ empty squares.  Disregarding these leaves $(\kappa-d)_{c-d}$ where the color $c-d$ is restricted only by the part length.  Together, ignoring the $dm$ necessarily empty squares of a composition of $n$ with colors $1, \ldots, d$ prohibited leaves an unrestricted $n$-color composition of $n-dm$ with $m$ parts.  By Theorem \ref{ccompparts}, there are $\binom{n-dm+m-1}{2m-1}$ of these.  Summing over possible $m$ values gives the formula.
\end{proof} 

It is worth mentioning that the $d=1$ case has recurrence $a(n) = 2a(n-1)$ and direct formula
\[\binom{n-1}{1} + \binom{n-1}{3} + \cdots = \binom{n-2}{0}+\binom{n-2}{1} + \binom{n-2}{2} + \binom{n-2}{3} + \cdots = 2^{n-2}\]
by the binomial theorem.  See also the notes for the $d=1$ case of the next proposition.

\begin{proposition} \label{prohibbij}
There is a bijection between
\begin{enumerate}
\item[(i.)] $n$-color compositions of $n+d$ prohibiting colors $1,2,\ldots,d$ and 
\item[(ii.)] regular compositions of $n$ with parts congruent to $1,2,\ldots,d \bmod 2d$ where any part $i$ for $2\leq i \leq d-1$ must either be the first part or be followed by an odd number of parts $d$.
\end{enumerate}
\end{proposition}

\begin{proof}
Apply the following map to the spotted tiling of an $n$-color composition of $n+d$ with colors $1,2,\ldots,d$ prohibited.
\begin{enumerate}
\item Remove the first $d$ squares (necessarily empty), changing the first part $\kappa_c$ to $(\kappa-d)_{c-d}$.
\item Convert each part $\kappa_c$ to $c$ followed by $k-c$ parts 1 in the regular composition, i.e., the $c$-block becomes a part $c$ and each square of the tail becomes a part 1.
\item In the resulting regular composition, replace any part $j$ such that $d+1\leq j \leq 2d \bmod 2d$ by the parts $j-d, d$.
\end{enumerate}

Note that step (1) makes a composition with sum $n$ and step (3) guarantees that each part is congruent to $1,2,\ldots,d \bmod 2d$. In the resulting regular composition, the only time a part $i$ with $2\leq i \leq d-1$ can be generated is either from the conversion of the shortened part in (1), in which case $i$ is the first part,  or from $j-d$ in step (3), in which case it is followed by a $d$.  The only way for $i$ to be followed by additional parts $d$ is subsequent applications of (3) to $2d$ which each yield two parts $d$, thus $i$ is followed by an odd number of parts $d$.  

For the reverse map, begin with a specified regular composition of $n$. 

\begin{enumerate}
\item Working from right to left, every time we have a part $d$, we add it to the next part to the left  $\ell$ to make a part $\ell + d$.  Also, change the leftmost part $k$ to $k+d$.
\item Starting again from the right, merge runs of parts 1 (possibly empty) with the larger part to the left into $n$-colored parts.  Specifically, a length $b$ run of parts 1 and the part $d+\ell$ to its left becomes the part $(\ell+d+b)_{\ell+d}$, that is, a $(\ell+d)$-block with tail length $b$.
\end{enumerate}

By the restriction on the compositions given in the proposition statement, step (1) produces compositions of $n+d$ with each part 1 or greater than $d$ whose first part is not 1.  Step (2) then converts these into $n$-color compositions with colors $1, 2, \ldots, d$ prohibited.  See Figure 5
for an example of the bijection.
\end{proof}

\begin{figure}[h]
\setlength{\unitlength}{.4cm}
\begin{picture}(18,4)
\thicklines

\put(0,3){\line(1,0){18}}
\put(0,4){\line(1,0){18}}
\put(0,3){\line(0,1){1}}
\put(3.5,3.5){\circle*{.5}}
\put(5,3){\line(0,1){1}}
\put(7.5,3.5){\circle*{.5}}
\put(8,4){\line(0,-1){1}}
\put(11.5,3.5){\circle*{.5}}
\put(12,4){\line(0,-1){1}}
\put(15.5,3.5){\circle*{.5}}
\put(18,4){\line(0,-1){1}}

\put(2,1){\line(1,0){16}}
\put(2,2){\line(1,0){16}}
\put(2,1){\line(0,1){1}}
\put(4,1){\line(0,1){1}}
\put(5,1){\line(0,1){1}}
\put(6,1){\line(0,1){1}}
\put(8,1){\line(0,1){1}}
\put(10,1){\line(0,1){1}}
\put(12,1){\line(0,1){1}}
\put(16,1){\line(0,1){1}}
\put(17,1){\line(0,1){1}}
\put(18,1){\line(0,1){1}}

\end{picture}
\caption{An example of the bijection in Proposition \ref{prohibbij} for $n=16$ and $d=2$, which associates the spotted tiling of the $n$-colored composition $(5_3, 3_3, 4_4,6_4)$ and the regular composition tiling for $(2,1,2,2,2,4,1,1)$ below it.}\label{fig:bij1c}
\label{fig:bij8}
\end{figure}
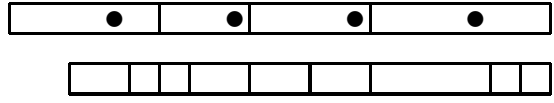

In the case of $d=1$, the restriction about parts $i$ in the regular compositions is vacuous; the bijection is between $n$-color compositions of length $n+1$ with color 1 prohibited and all regular compositions of length $n$.  In the first mapping, step (3) does not arise.  In the reverse map, ignore step (1) except to add an initial part 1 to the composition.  See Figure \ref{fig:d1bij} for the complete $n=3$, $d=1$ case.

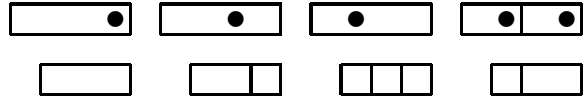
\begin{figure}[h]
\setlength{\unitlength}{.4cm}
\begin{picture}(19,4)
\thicklines

\put(0,3){\line(1,0){4}}
\put(0,4){\line(1,0){4}}
\put(0,3){\line(0,1){1}}
\put(3.5,3.5){\circle*{.5}}
\put(4,3){\line(0,1){1}}

\put(1,1){\line(1,0){3}}
\put(1,2){\line(1,0){3}}
\put(1,1){\line(0,1){1}}
\put(4,1){\line(0,1){1}}

\put(5,3){\line(1,0){4}}
\put(5,4){\line(1,0){4}}
\put(5,3){\line(0,1){1}}
\put(7.5,3.5){\circle*{.5}}
\put(9,3){\line(0,1){1}}

\put(6,1){\line(1,0){3}}
\put(6,2){\line(1,0){3}}
\put(6,1){\line(0,1){1}}
\put(8,1){\line(0,1){1}}
\put(9,1){\line(0,1){1}}

\put(10,3){\line(1,0){4}}
\put(10,4){\line(1,0){4}}
\put(10,3){\line(0,1){1}}
\put(11.5,3.5){\circle*{.5}}
\put(14,3){\line(0,1){1}}

\put(11,1){\line(1,0){3}}
\put(11,2){\line(1,0){3}}
\put(11,1){\line(0,1){1}}
\put(12,1){\line(0,1){1}}
\put(13,1){\line(0,1){1}}
\put(14,1){\line(0,1){1}}

\put(15,3){\line(1,0){4}}
\put(15,4){\line(1,0){4}}
\put(15,3){\line(0,1){1}}
\put(16.5,3.5){\circle*{.5}}
\put(17,3){\line(0,1){1}}
\put(18.5,3.5){\circle*{.5}}
\put(19,3){\line(0,1){1}}

\put(16,1){\line(1,0){3}}
\put(16,2){\line(1,0){3}}
\put(16,1){\line(0,1){1}}
\put(17,1){\line(0,1){1}}
\put(19,1){\line(0,1){1}}

\end{picture}
\caption{The complete bijection in Proposition \ref{prohibbij} for $n=3$ and $d=1$, which associates the spotted tilings of $(4_4)$, $(4_3)$, $(4_2)$, $(2_2,2_2)$ with the regular tilings of $(3)$, $(2,1)$, $(1,1,1)$, $(1,2)$, respectively.}\label{fig:d1bij}
\label{fig:bij8}
\end{figure}

\subsubsection{Prohibiting color 2}

We conclude this subsection with an illustration of the specific results that can be found for particular restrictions, here, prohibiting the color 2.  See \cite{hw} for further examples of this type.  We begin with enumeration results, including a combinatorial proof for a formula listed in \cite[A034943]{oeis}.

\begin{proposition} \label{no2enum}
The number of $n$-color compositions of $n$ prohibiting color 2 is given by
\[a(n) = 3a(n-1) - 2a(n-2) + a(n-3)\]
with $a_0 = a_1 = 1$ and $a_2 = 2$.  Also,
$$ \sum_{k=0}^{\lfloor n/2 \rfloor} {n+k \choose 3k} .$$
\end{proposition}

\begin{proof}
The recurrence follows directly from Theorem \ref{prohibited}.

For the direct formula, we show that there are $ {n+k \choose 3k}$ compositions of $n$ prohibiting color 2 that have $k$ parts greater than 1.  Summing over possible values of $k$, from $0$ to $\lfloor $n/2$ \rfloor$, will establish the formula.
 
Here is a bijection between the color compositions of $n$ prohibiting color 2 and with $k$ greater than 1, and $1\times (n+k)$ rectangles with $3k$ marked squares.

Given a $1\times (n+k)$ rectangle, label its squares $1, 2, \ldots , n+k$ with marked squares in positions $a_1 < a_2 <  \cdots < a_{3k}$.
For $i=1, \ldots, k$, form a spotted tile from the squares $a_{3i-2}, \ldots, a_{3i-1}, \ldots, a_{3i}$
as follows.
\begin{enumerate}
\item Let $a_{3i-1}$ be the spotted square.
\item[2a.] If $a_{3i-1}=a_{3i-2}+1$, remove the square $a_{3i-2}$.
\item[2b.] If $a_{3i-1}\neq a_{3i-2}+1$, remove the square $a_{3i}$.
\item[3.] Map any square not in $\{ a_{3i-2}, \ldots, a_{3i-1}, \ldots, a_{3i} \}$ for any $i$ to a spotted tile corresponding to $1_1$.
\end{enumerate}
Steps 1 and 2 create a spotted tile corresponding to a color composition part of length $a_{3i} - a_{3i-2}$ (which is at least 2) with color $a_{3i-1} - a_{3i-2}+1$ unless that would give a 2, in which case the color is 1.  Removing $k$ squares and erasing marked squares that did not become spots leaves a desired color composition of $n$.  See Figure \ref{fig:bij_n1} for an example.

For the reverse map, given a color composition of $n$ prohibiting color 2 with $k$ parts greater than 1, we modify each of those $k$ parts.
\begin{enumerate}
\item[1a.] Change $\kappa_1$ with $\kappa \ge 2$ to $(\kappa+1)_2$.  
\item[1b.] Change $\kappa_c$ with $\kappa \ge 2$ and $c \ge 2$ to $(\kappa+1)_c$.
\item[2.] Mark the first square, spotted square, and last square of these resulting $k$ parts.
\item[3.] Erase the spots of any parts $1_1$.
\item[4.] Erase any lines separating parts.
\end{enumerate}
Note that step 1 produces parts with no color 1 and with each part having a positive length tail (i.e., no last square is spotted).  Extending each of these $k$ parts by one gives a color composition of $n+k$.  Also, for each of these $k$ parts, step 2 described three distinct squares.  Steps 2 through 4 give a $1 \times (n+k)$ rectangle with $3k$ marked squares.
\end{proof}

\begin{figure}[t]
\setlength{\unitlength}{.4cm}
\begin{picture}(14,10)(-3,0.5)
\thicklines

\put(-6,9){\line(1,0){20}}
\put(-6,10){\line(1,0){20}}
\put(-4.5,9.5){\circle{.5}}
\put(-6,9){\line(0,1){1}}
\put(-5.5,9.5){\circle{.5}}
\put(-1.5,9.5){\circle{.5}}
\put(.5,9.5){\circle{.5}}
\put(2.5,9.5){\circle{.5}}
\put(3.5,9.5){\circle{.5}}
\put(5.5,9.5){\circle{.5}}
\put(6.5,9.5){\circle{.5}}
\put(8.5,9.5){\circle{.5}}
\put(9.5,9.5){\circle{.5}}
\put(11.5,9.5){\circle{.5}}
\put(13.5,9.5){\circle{.5}}
\put(14,9){\line(0,1){1}}

\put(-6,7){\line(1,0){20}}
\put(-6,8){\line(1,0){20}}
\put(-4.5,7.5){\circle{.5}}
\put(-6,7){\line(0,1){1}}
\put(-5.5,7.5){\circle{.5}}
\put(-1.5,7.5){\circle{.5}}
\put(-1,7){\line(0,1){1}}
\put(-.5,7.5){\circle*{.5}}
\put(0,7){\line(0,1){1}}
\put(.5,7.5){\circle{.5}}
\put(2.5,7.5){\circle{.5}}
\put(3.5,7.5){\circle{.5}}
\put(4,7){\line(0,1){1}}
\put(4.5,7.5){\circle*{.5}}
\put(5,7){\line(0,1){1}}
\put(5.5,7.5){\circle{.5}}
\put(6.5,7.5){\circle{.5}}
\put(9,8){\line(0,-1){1}}
\put(8.5,7.5){\circle{.5}}
\put(9.5,7.5){\circle{.5}}
\put(11.5,7.5){\circle{.5}}
\put(13.5,7.5){\circle{.5}}
\put(14,7){\line(0,1){1}}

\put(-6,5){\line(1,0){20}}
\put(-6,6){\line(1,0){20}}
\put(-4.5,5.5){\circle*{.5}}
\put(-6,5){\line(0,1){1}}
\put(-5.5,5.5){\circle{.5}}
\put(-1.5,5.5){\circle{.5}}
\put(-1,5){\line(0,1){1}}
\put(-.5,5.5){\circle*{.5}}
\put(0,5){\line(0,1){1}}
\put(.5,5.5){\circle{.5}}
\put(2.5,5.5){\circle*{.5}}
\put(3.5,5.5){\circle{.5}}
\put(4,5){\line(0,1){1}}
\put(4.5,5.5){\circle*{.5}}
\put(5,5){\line(0,1){1}}
\put(5.5,5.5){\circle{.5}}
\put(6.5,5.5){\circle*{.5}}
\put(9,6){\line(0,-1){1}}
\put(8.5,5.5){\circle{.5}}
\put(9.5,5.5){\circle{.5}}
\put(11.5,5.5){\circle*{.5}}
\put(13.5,5.5){\circle{.5}}
\put(14,5){\line(0,1){1}}

\put(-5,3){\line(1,0){8}}
\put(-5,4){\line(1,0){8}}
\put(-4.5,3.5){\circle*{.5}}
\put(-5,3){\line(0,1){1}}
\put(-1,3){\line(0,1){1}}
\put(-.5,3.5){\circle*{.5}}
\put(0,3){\line(0,1){1}}
\put(2.5,3.5){\circle*{.5}}
\put(3,3){\line(0,1){1}}

\put(4,3){\line(1,0){1}}
\put(4,4){\line(1,0){1}}
\put(4,3){\line(0,1){1}}
\put(4.5,3.5){\circle*{.5}}
\put(5,3){\line(0,1){1}}

\put(6,3){\line(0,1){1}}
\put(6,3){\line(1,0){7}}
\put(6,4){\line(1,0){7}}
\put(6.5,3.5){\circle*{.5}}
\put(9,4){\line(0,-1){1}}
\put(11.5,3.5){\circle*{.5}}
\put(13,3){\line(0,1){1}}

\put(-4,1){\line(1,0){16}}
\put(-4,2){\line(1,0){16}}
\put(-4,1){\line(0,1){1}}
\put(-3.5,1.5){\circle*{.5}}
\put(0,1){\line(0,1){1}}
\put(1,1){\line(0,1){1}}
\put(.5,1.5){\circle*{.5}}
\put(4,1){\line(0,1){1}}
\put(4.5,1.5){\circle*{.5}}
\put(5,1){\line(0,1){1}}
\put(3.5,1.5){\circle*{.5}}
\put(8,2){\line(0,-1){1}}
\put(12,2){\line(0,-1){1}}
\put(5.5,1.5){\circle*{.5}}
\put(10.5,1.5){\circle*{.5}}

\end{picture}
\caption{An example of the bijection of Proposition \ref{no2enum}.  The $1 \times 20$ rectangle with indicated 12 marked squares corresponds to the color composition $(4_1, 1_1, 3_3,1_1, 3_1,4_3)$ of 16 which has four parts greater than 1.}
\label{fig:bij_n1}
\end{figure}
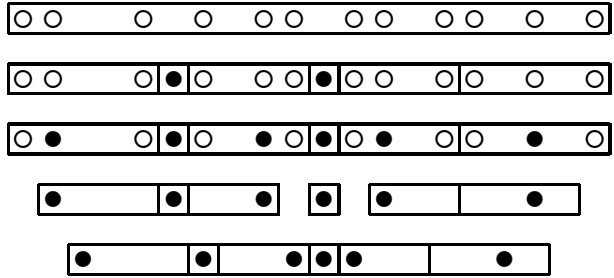

Next, we prove a connection between $n$-color compositions prohibiting the color 2 and certain regular compositions.  To present this bijection, we introduce the notions of {\it left open}, {\it right open},  and {\it open} parts in a composition. A left open part of size $k$, denoted by $\cdot k$, is a part of size $k$ that can be joined with its adjacent part (on the left) if that part is right open or open; the definition for right open part $k\cdot$ or open part $\cdot k \cdot$ are similar. Standard parts may be called closed for clarity.  For example, 
$ (3 \cdot, \cdot 2 \cdot, 4 \cdot, \cdot 3, \cdot 1 \cdot, \cdot 5)$
corresponds to the regular composition $(5,7,6)$ where 5 is a result of $3 \cdot $ and $ \cdot 2 \cdot $, 7 is a result of $4 \cdot $ and $ \cdot 3$, and 6 is a result of $ \cdot 1 \cdot $ and $ \cdot 5$; see Figure \ref{fig:ex_open}.

\begin{figure}[h]
\setlength{\unitlength}{.4cm}
\begin{picture}(21,4)
\thicklines

\put(0,3){\line(1,0){3}}
\put(0,4){\line(1,0){3}}
\put(0,3){\line(0,1){1}}

\put(4,3){\line(1,0){2}}
\put(4,4){\line(1,0){2}}

\put(7,3){\line(1,0){4}}
\put(7,4){\line(1,0){4}}
\put(7,3){\line(0,1){1}}

\put(12,3){\line(1,0){3}}
\put(12,4){\line(1,0){3}}
\put(15,3){\line(0,1){1}}

\put(16,3){\line(1,0){1}}
\put(16,4){\line(1,0){1}}

\put(18,3){\line(1,0){5}}
\put(18,4){\line(1,0){5}}
\put(23,3){\line(0,1){1}}

\put(3,1){\line(1,0){18}}
\put(3,2){\line(1,0){18}}
\put(3,1){\line(0,1){1}}
\put(8,1){\line(0,1){1}}
\put(15,1){\line(0,1){1}}
\put(21,1){\line(0,1){1}}

\end{picture}
\caption{A representation of $(3 \cdot, \cdot 2 \cdot, 4 \cdot, \cdot 3, \cdot 1 \cdot, \cdot 5)$ leading to $(5,7,6)$.}\label{fig:ex_open}
\end{figure}
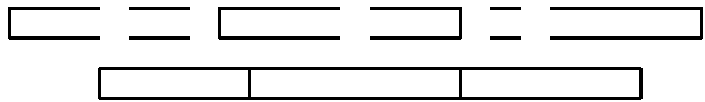

\begin{proposition}
There is a bijection between 
\begin{itemize}
\item[(i.)] $n$-color compositions of $n$ with color 2 prohibited and
\item[(ii.)] regular compositions of $3n+2$ with parts congruent to 2 modulo 3.
\end{itemize}
\end{proposition}

\begin{proof}
Note that the regular compositions described in (ii) must each have $3m+1$ parts for some nonnegative integer $m$ in order for the parts congruent to 2 modulo 3 to sum to $3n+2$.

Given a color composition of $n$ with color 2 prohibited, we triple each part size and build a regular composition as follows.
\begin{enumerate}
\item Change any $1_1$ to $\cdot  3 \cdot$ (an open 3).
\item[2a.] Change any $\kappa_1$ with $\kappa \ge 2$ to $ \cdot 2, 2, 3\kappa-4$ (a left open 2 followed by a closed 2 and a closed $3\kappa-4$).
\item[2b.] Change any $\kappa_c$ with $\kappa \ge 2$ and $c \ge 3$ to $\cdot 2, 3c-4, 3\kappa-3c+2$.
\item[3.] Add $\cdot 2$ at the end of the composition.
\item[4.] Combine any open and left open parts as described in their definition.
\end{enumerate}
There are several characteristics of the map to verify.  

Since each part of the color composition is tripled in length and step 3 adds a length 2 part, the resulting composition sums to $3n+2$.

The parts are all 2 modulo 3: The closed parts generated in step 2 all have that form, and any run of $\cdot 3 \cdot$ terms (possibly empty) combines with a subsequent $\cdot 2$ (guaranteed by step 3) to give a part of length 2 modulo 3.

Finally, for the constraint on the number of parts, suppose that there are $m$ parts in the color composition of length 2 or more.  For each of these $m$ parts, step 2 creates 3 parts: an initial part which is the merger of any previous $\cdot 3 \cdot$ terms with $\cdot 2$ and two closed parts specified by (2a) or (2b).  Step 3 results in one additional part, giving $3m+1$ in total.

For the reverse map, consider a regular composition of $3n+2$ with parts congruent to 2 modulo 3 that has $3m+1$ parts.
\begin{enumerate}
\item For the first part and every third part after that, i.e., for each of the $(3k+1)$st parts for $k=0, \ldots , m$, change it from $3a+2$ to $a$ copies of $1_1$ followed by a 2;
\item Remove the final 2.
\item Every remaining 2 from step 1 is followed by the $(3k+2)$nd and $(3k+3)$rd parts of the regular composition. Write these three consecutive parts as $(2, x, y)$. 
\item[3a.] If $x=2$, change $(2,x,y)$ to $((2+x+y)/3)_1$;
\item[3b.] If $x>2$, change $(2,x,y)$ to $((2+x+y)/3)_ {(x+4)/3}$.
\end{enumerate}
Step 2 is valid since the last part is the $(3m+1)$st which was converted in step 1 to a sequence ending in 2.  The various part lengths and colors in step 3 are all integers since $x, y \equiv 2 \bmod 3$.  

For the composition sums, each application of step 3 replaces $2+x+y$ towards the sum of the regular composition with a part $(2+x+y)/3$ in the color composition, and step 1 takes a $3a$ contribution to the regular composition to $a$ copies of $1_1$, so that the color composition has sum $n$.  

Note that the colors assigned are either 1 or $(x+4)/3$ with $x>2$, i.e., a color 3 or greater.  

Towards seeing that this is the inverse of the previous map, step 3 applies to $m$ triples $(2,x,y)$ so that the resulting color composition has $m$ parts of length 2 or greater.
\end{proof}

For example, the color composition $(2_1,1_1,1_1,3_3,4_4,1_1)$ of 12, with 3 parts 2 or greater, corresponds the regular composition $(2,2,2,8,5,2,2,8,2,5)$ of 38, with $3\cdot 3+1 = 10$ parts, as follows.

\begin{center}
{\renewcommand{\arraystretch}{1.5}
\begin{tabular}{rcl}
$ (2_1,1_1,1_1,3_3,4_4,1_1)$ & $\longrightarrow$ & $(\cdot 2,2,2; \cdot 3 \cdot; \cdot 3 \cdot; \cdot 2,5,2; \cdot 2,8,2; \cdot 3 \cdot; \cdot 2)$ \\
& $\longrightarrow$ & $(2,2,2,8,5,2,2,8,2,5)$ \\ \hline
$(\underline{2},2,2,\underline{8},5,2,\underline{2},8,2,\underline{5})$ & $\longrightarrow$ & $([2,2,2],1_1,1_1,[2, 5, 2], [2, 8, 2], 1_1, 2)$ \\
& $\longrightarrow$ & $(2_1,1_1,1_1,3_3,4_4,1_1)$ 
\end{tabular}}
\end{center}

\subsection{Colors under modular conditions}

Certain cases of modular color conditions have been considered recently.  Odd and even colors are aspects of many combined restrictions in \cite{sa}.  The focus of \cite{s19} is allowing colors within a singular modular class $i \bmod m$.  In this section, we give results that expand on this work.

\subsubsection{Compositions with colors of a single modular congruence}

Most of \cite{s19} considers $n$-color compositions with colors restricted to those of the form $am + i$ for positive $m,i$ and $a \ge 0$.  This is slightly more general than our $i \mod m$ since, for example, $m = 2$ and $i = 5$ gives allowed colors $5, 7, 9, \ldots$ rather than all odd colors.  Acosta et al. give further results for the case $i = m+1$, including a combinatorial proof involving binary words \cite[Thm. 10]{s19}.  We give a related proof for all $i$ in the range $2 \le i \le m$.

\begin{proposition} \label{binseq}
Given a modulus $m \ge 2$ and an $i$ such that $2 \le i \le m$, there is a bijection between 
\begin{itemize}
\item[(i.)] $n$-color compositions of $n$ with colors congruent to $i \bmod m$ and
\item[(ii.)] length $n-1$ binary strings that start with 1 where runs of 1's have length congruent to $i-1 \mod m$.
\end{itemize}
\end{proposition}

\begin{proof}
Given a color composition of $n$ with all colors congruent to $i \bmod m$, 
\begin{itemize}
\item[1a.] In each part, map the empty squares of the $c$-block (i.e., the squares before the spot) to 1.
\item[1b.] In each part, map the spotted square and tail to 0.
\item[2.] Remove a 0 at the end.
\end{itemize}
Since allowed colors satisfy $c \ge 2$, the first binary digit is 1.  Because $c \equiv i \bmod m$ and each run of 1's begins with a corresponding part, the length of each run of 1's is congruent to $i-1 \mod m$.  And since the tiling for a color composition ends in a spot or empty square after a spot, the last binary digit is 0 (before it is removed in step 2).

The reverse map is clear: Given a specified binary string, add a 0 at the end and convert alternating runs of 1's and 0's into tilings of a color composition.
\end{proof}

See Figure \ref{fig:bij5} for an example with even colors.  We note that the proposition statement holds for $i=1$, but neither the combinatorial proof here nor the related one in \cite{s19} cover that case.

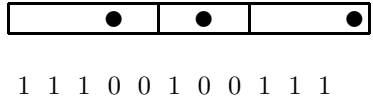
\begin{figure}[h]
\setlength{\unitlength}{.4cm}
\begin{picture}(12,4)
\thicklines

\put(0,3){\line(1,0){12}}
\put(0,4){\line(1,0){12}}
\put(0,3){\line(0,1){1}}
\put(5,3){\line(0,1){1}}
\put(3.5,3.5){\circle*{.5}}
\put(8,4){\line(0,-1){1}}
\put(12,4){\line(0,-1){1}}
\put(6.5,3.5){\circle*{.5}}
\put(11.5,3.5){\circle*{.5}}

\put(0.3,1){1}
\put(1.3,1){1}
\put(2.3,1){1}
\put(3.3,1){0}
\put(4.3,1){0}

\put(5.3,1){1}
\put(6.3,1){0}
\put(7.3,1){0}

\put(8.3,1){1}
\put(9.3,1){1}
\put(10.3,1){1}

\end{picture}
\caption{The bijection of Proposition \ref{binseq} for $m = i = 2$ where the composition $(5_4, 3_2, 4_4)$ of 12 corresponds to the length 11 binary string $11100\;100\;111$ (spaces corresponding to parts are only for clarity).}
\label{fig:bij5}
\end{figure}

\subsubsection{Connecting compositions with odd colors and with no parts 1}
The first new restriction considered in \cite{sa} is $n$-color compositions with just odd colors \cite[A006054]{oeis}.  Sachdeva and Agarwal also consider a restriction due to Guo \cite{g12}, $n$-color compositions where the part $1_1$ is prohibited \cite[A219788]{oeis}.  We conclude with a new connection between these two types of $n$-color compositions.

\begin{proposition} \label{oddno1}
Let $a(n)$ be the number of $n$-color compositions of $n$ with only odd colors allowed and let $b(n)$ be the number of $n$-color compositions of $n$ where the part $1_1$ is prohibited.  Then $a(n) = b(n) + b(n-1)$.
\end{proposition}

\begin{proof}
Write $A(n)$ for the set of color compositions of $n$ with only odd colors allowed and $B(n)$ for the set of color compositions of $n$ where the part $1_1$ is prohibited.  We establish a bijection between $A(n)$ and  $B(n) \cup B(n-1)$.

Starting from $B(n) \cup B(n-1)$,
\begin{enumerate}
\item Create the set $B'(n)$ by adding a $1_1$ at the end of each composition in $B(n-1)$.
\item In the compositions of $B(n) \cup B'(n)$, change each even part $\kappa_{2c}$ to $1_1, (\kappa-1)_{2c-1}$.
\end{enumerate}
Note that $B'(n)$ from step 1 is disjoint from $B(n)$ which contains no parts $1_1$.  An alternative description of step 2 is to make the first square (empty since the color is even) spotted corresponding to a new part; note that the spot for $\kappa_{2c}$ stays in the same position.  See Figure \ref{fig:odd} for an example.  Clearly, step 2 results in compositions with only odd colors.

For the reverse map, given a composition in $A(n)$,
\begin{enumerate}
\item Working left to right, any $1_1$ having a successor $\kappa_{2c-1}$ becomes the part $(\kappa+1)_{2c}$.
\item Any terminal $1_1$ is removed.
\end{enumerate}
Note that the spot indicating the color of $\kappa_{2c-1}$ does not move in the merger giving $(\kappa+1)_{2c}$.  Clearly, the procedure leaves no parts $1_1$.  The compositions with a $1_1$ at the end become elements of $B(n-1)$; those with other terminal parts become elements of $B(n)$.
\end{proof}

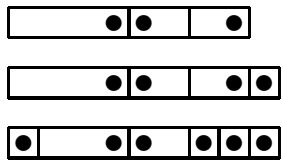
\begin{figure}[h]
\setlength{\unitlength}{.4cm}
\begin{picture}(12,6)
\thicklines

\put(0,5){\line(1,0){8}}
\put(0,6){\line(1,0){8}}
\put(3.5,5.5){\circle*{.5}}
\put(4.5,5.5){\circle*{.5}}
\put(7.5,5.5){\circle*{.5}}
\put(0,5){\line(0,1){1}}
\put(4,5){\line(0,1){1}}
\put(6,5){\line(0,1){1}}
\put(8,5){\line(0,1){1}}

\put(0,3){\line(1,0){9}}
\put(0,4){\line(1,0){9}}
\put(3.5,3.5){\circle*{.5}}
\put(4.5,3.5){\circle*{.5}}
\put(7.5,3.5){\circle*{.5}}
\put(8.5,3.5){\circle*{.5}}
\put(0,3){\line(0,1){1}}
\put(4,3){\line(0,1){1}}
\put(6,3){\line(0,1){1}}
\put(8,3){\line(0,1){1}}
\put(9,3){\line(0,1){1}}

\put(0,1){\line(1,0){9}}
\put(0,2){\line(1,0){9}}
\put(0.5,1.5){\circle*{.5}}
\put(3.5,1.5){\circle*{.5}}
\put(4.5,1.5){\circle*{.5}}
\put(6.5,1.5){\circle*{.5}}
\put(7.5,1.5){\circle*{.5}}
\put(8.5,1.5){\circle*{.5}}
\put(0,1){\line(0,1){1}}
\put(1,1){\line(0,1){1}}
\put(4,1){\line(0,1){1}}
\put(6,1){\line(0,1){1}}
\put(7,1){\line(0,1){1}}
\put(8,1){\line(0,1){1}}
\put(9,1){\line(0,1){1}}

\end{picture}
\caption{The bijection of Proposition \ref{oddno1} for $n=9$ where $(4_4,2_1,2_2) \in B(8)$ corresponds to $(1_1,3_3,2_1,1_1,1_1,1_1) \in A(9)$.}\label{fig:odd}
\end{figure}

\end{document}